\documentclass[12pt]{article}
\usepackage[margin=2in,includefoot,footskip=12pt,]{geometry}
\usepackage{amsfonts,amsmath, amsthm, amssymb,hyperref,fullpage}
\usepackage{graphicx}
\usepackage{listings}
\usepackage{xcolor,enumitem}
\usepackage{tikz}
\usepackage{algorithm}
\usepackage{tcolorbox}
\usepackage[noend]{algpseudocode}
\usetikzlibrary{arrows}
\usetikzlibrary{graphs,graphs.standard}
\usetikzlibrary{shapes.geometric}
\usetikzlibrary{svg.path}
\usepackage{bm}
\usetikzlibrary{patterns}

\newtheorem{theorem}{Theorem}[section]

\newtheorem{cor}[theorem]{Corollary}

\newtheorem{remark}[theorem]{Remark}

\DeclareMathOperator{\diam}{diam}
\DeclareMathOperator{\tr}{Tr}

\newcommand{\affl}[3]{%
\noindent #1,
\textsc{#3}\\
Email: \texttt{#2}\\[1.5pt]
}

\title{ On the $\ell_\infty$-analog of Algebraic Connectivity}

\author{ M. Rajesh Kannan,  Rahul Roy}

\date{\today}

\begin{document}
\maketitle
\begin{abstract} 

The algebraic connectivity of a graph, defined as the second smallest eigenvalue of its Laplacian matrix, admits a well-known variational characterization involving the $\ell_2$-norm. Motivated by the recent introduction of its $\ell_\infty$-analog by Andrade and Dahl, we investigate the graph parameter $\gamma(G)$, obtained by replacing the $\ell_2$-norm with the $\ell_\infty$-norm in the corresponding optimization problem. We establish a simple and explicit combinatorial formula expressing $\gamma(G)$ as the ratio of the order of the graph to its maximum transmission, thereby providing a direct graph-theoretic interpretation of the parameter. We prove that $\gamma(G)$ characterizes graph connectivity and completely characterize all $\ell_\infty$-Fiedler vectors as the vectors
\[
\left\{\pm\left(1-\gamma(G)d(u,\cdot)\right):u\in \mathcal{M}(G)\right\},
\]
where $\mathcal{M}(G)$ denotes the set of vertices of maximum transmission. We also establish a product formula for $\gamma(G)$ under Cartesian products, yielding explicit expressions for hypercubes, Hamming graphs, and torus grid graphs. Furthermore, we establish comparisons involving $\gamma(G)$, the distance spectral radius, the $\ell_1$-smoothing parameter, and the algebraic connectivity, and investigate equality in these bounds. For connected non-complete strongly regular graphs, equality in our algebraic-connectivity bound holds precisely when the least adjacency eigenvalue is $-2$. Finally, we extend the transmission formula and the characterization of optimal vectors to weighted graphs.

\end{abstract}


\section{Introduction}
The Laplacian matrices of graphs, along with their eigenvalues and eigenvectors, have found numerous applications in various domains, including combinatorial optimization, electrical networks, graph signal processing, data science, and machine learning. The second smallest eigenvalue of the Laplacian matrix provides a robust measure of the connectivity of the associated graph \cite{alg-con}.  It also serves as a lower or upper bound for several important graph parameters, such as the isoperimetric number, maximum cut, independence number, genus, and diameter \cite{mohar-lap-sur-1, mohar-eigen-comb-opti, trevisan_spec-partition}. The corresponding eigenvector is widely used in graph partitioning problems and data clustering, a method known as spectral clustering \cite{trevisan_spec-partition, ulrike-spectral}. In his seminal paper \cite{alg-con}, Fiedler coined the term algebraic connectivity for the second smallest eigenvalue of the Laplacian matrix. The algebraic connectivity of a graph $G$, denoted by $a(G)$,  has an interesting optimization characterization involving the $\ell_2$-norm in the objective function and the constraints. Recently, in \cite{enide-geir-2024}, Andrade and Dahl studied the counterparts of this optimization problem with respect to the $\ell_1$-norm and the $\ell_\infty$-norm. The corresponding optimal values are denoted by $b(G)$ and $\gamma (G)$, respectively. They showed that $b(G)$ has an interesting connection to the sparsest cut problem, implying that computing $b(G)$ is an NP-hard problem. They also proposed a linear programming formulation to compute $\gamma (G)$. More recently, Fern\'andez Goycoolea et al.\
\cite{FernandezGoycooleaEtAl2026} presented a geometric treatment
of the $\ell_1$-minimization problem and its corresponding
maximization problem. They showed that the optimal value of the
maximization problem is the average of the two largest vertex
degrees.

In this article, we provide an explicit combinatorial formula for $\gamma (G)$ in terms of the maximum row sum of the distance matrix of $G$. This leads to a polynomial-time combinatorial algorithm for computing $\gamma (G)$. Further, we study the connections between  $\gamma (G)$ and various graph parameters associated with $G$. In addition, we derive a formula for $\gamma(G)$ when $G$ is the Cartesian product of finitely many graphs.

Let $G = (V(G), E(G))$ be an unweighted, undirected, and simple graph, where $V(G)$ is the set of vertices and $E(G)$ is the set of edges of $G$. Let the number of vertices in $G$, denoted by $\vert V(G) \vert $, be $n$. If the vertex $ u $ is adjacent to the vertex $v$, then we write $ u \sim v $.  The \textit{adjacency matrix} of a simple graph $ G $, denoted by $ A(G) $,  is the symmetric  $ n \times n $  matrix whose  $ (u,v)$-th entry is defined by $ a_{uv}=1 $ if $ u \sim v $, and  $ a_{uv} = 0 $ otherwise. The \textit{Laplacian matrix} $L(G)$ is the $n\times n$  matrix defined as 
$$
L(G)=\Delta(G)-A(G),
$$
where $\Delta(G)$ is the \textit{diagonal matrix} with the degrees of the vertices of $G$ as its diagonal elements, and $A(G)$ is the \textit{adjacency matrix} of $G$. The eigenvalues of adjacency matrices can be used to obtain various structural properties of graphs, such as bipartiteness, regularity, and more. Additionally, the eigenvalues of adjacency matrices provide bounds for several graph parameters, including the chromatic number, clique number, independence number, and others\cite{graphs-and-matrices, ratio-bound,  nica-book-2018, clique-indep-bound-wilf}.
On the other hand, the eigenvalues of the Laplacian matrix of a graph reveal information about the connectedness of the graph $G$. In particular, $0$ is always an eigenvalue of $L(G)$, with the all-ones vector $\mathbf{1}=(1,\dots,1)^{\top}\in\mathbb{R}^{n}$ as a corresponding eigenvector. The second smallest eigenvalue of $L(G)$, denoted by $a(G)$, is called the \textit{algebraic connectivity} of the graph $G$ \cite{alg-con}. It is well known that $G$ is connected if and only if $a(G) > 0$ \cite{fiedler-1989}.
An eigenvector corresponding to $a(G)$ is called a \textit{Fiedler vector} \cite{graphs-and-matrices}. The optimization approach to computing $a(G)$ is well known and is based on the Courant-Fischer theorem \cite{matrix-analysis}. An explicit formula for $a(G)$ is given as follows:
 \begin{equation*}
a(G)=\min  \biggl\{\sum_{uv \in E(G)}(x_u - x_v)^2: \sum_{v \in V(G)} x_v =0 ~\mbox{and}~\Vert x \Vert _{2}=1\ \biggl\},
 \end{equation*} 
 where $\Vert x \Vert_2 := \sqrt{\sum\limits_{i=1}^n \vert x_i \vert^2}$ is the $\ell_2$-norm of the vector $x\in \mathbb{R}^n$. Note that the vector $x$ that attains the minimum in the above optimization problem has the following property: the sum of the squared differences across the edges is minimized subject to the constraints. This can be viewed as "an optimal variation" along the edges of $G$. For this reason, Andrade and Dahl \cite{enide-geir-2024} referred to this as the $\ell_2$-smoothing problem. Furthermore, they investigated analogous problems under the $\ell_1$-norm and the $\ell_\infty$-norm, considering both the objective function and the constraint \cite{enide-geir-2024}. This led to the introduction of two new parameters, defined as follows:
$$
b(G)=\min\biggl\{\sum_{uv \in E(G)}\vert x_u - x_v\vert: \sum_{v \in V(G)}x_v=0~\mbox{and}~ \Vert x \Vert_{1}=1\biggl\},
$$
$$
\gamma(G)=\min\biggl\{\max_{uv \in E(G)}\vert x_u - x_v\vert: \sum_{v \in V(G)} x_v =0 ~\mbox{and}~\Vert x \Vert _{\infty}=1\biggl\}.
$$
Here, $\Vert x \Vert_1 := \sum\limits_{i=1}^n |x_i|$ denotes the $\ell_1$-norm, and $\Vert x \Vert_\infty := \max \{ |x_i| : 1 \leq i \leq n \}$ denotes the $\ell_\infty$-norm of the vector $x \in \mathbb{R}^n$. These optimization problems are well-defined, as the minima are attained due to the compactness of the constraint sets and the continuity of the objective functions. The vectors that give the optimal solutions are referred to as the {$\ell_1$-\textit{Fiedler vectors} and the {$\ell_\infty$-\textit{Fiedler vectors}, respectively. Furthermore, they formulated the following linear programming problem to compute  $\gamma(G)$. Let $G$ be a graph with $n$ vertices. For every $k$ with $1 \leq k \leq n$, consider the linear programming problem $LP(k)$:
\begin{align*}
\text{Minimize} \quad & y \\
\text{subject to} \quad & x_i - x_j \leq y \quad  (ij \in E(G)), \\
& -(x_i - x_j) \leq y \quad (ij \in E(G)), \\
& \sum_{j=1}^{n} x_j = 0, \\
& -1 \leq x_j \leq 1 \quad (1\leq j \leq n), \\
& x_k = 1.
\end{align*}
Solving $LP(k)$ for each $1 \leq k \leq n$, and choosing the minimum value of $y$ obtained gives the value of $\gamma(G)$. \\
The remainder of the paper is organized as follows. Section \ref{preliminaries} recalls the notation and preliminary results used throughout the paper. In Section \ref{main-results}, we derive the combinatorial formula for $\gamma(G)$, characterize all $\ell_\infty$-Fiedler vectors and use it to characterize the graphs having $\ell_{\infty}$-Fiedler vectors with zero coordinates. Section \ref{sec-product} treats Cartesian products and their applications. In Section \ref{compari-other}, we compare $\gamma(G)$ with the distance spectral radius, $b(G)$, and the algebraic connectivity, and study equality, including a characterization within the class of connected non-complete strongly regular graphs. Section \ref{weighted-analogue} extends the main formula and the description of optimal vectors to weighted graphs. Section \ref{conclusion} concludes with a problem concerning the structural classification of graphs attaining equality in Theorem \ref{alg-theorem}.
\section{Preliminaries}\label{preliminaries}
In this section, we recall some of the known definitions and results that will be used later.
For a vertex $u\in V(G)$, let $\deg(u)$ denote the degree of $u$ in $G$. A vertex $v$ of a graph $G$ is called a \textit{pendant vertex} if $\deg(v) =1$. The distance between two vertices $u$ and $v$ is denoted by $d(u,v)$. The \textit{diameter} of a graph $G$, denoted by $\diam(G)$, is the maximum distance between any pair of vertices in $G$, that is,
$$
\diam(G)=\max_{ u,v \in V(G)} d(u,v).
$$
The \textit{eccentricity} of a vertex $u \in V(G)$, denoted by $ecc(u)$, is the maximum distance from $u$ to any other vertex in the graph. That is,
$$
ecc(u)=\max\biggl\{d(u,v) :  v \in V(G)\biggl\}.
$$
Let $u \in V(G)$. Define 
$$
S_i(u)=\biggl\{v \in V(G) : d(u,v)=i \biggl\},
$$
for $0\leq i\leq ecc(u)$.
For a connected graph $G$, the distance matrix $\mathcal{D}(G)$ is the $n \times n$ matrix in which the $(u,v)$-entry $\mathcal{D}_{uv}$ is given by the distance $d(u,v)$ between the vertices $u$ and $v$. The spectrum of $\mathcal{D}(G)$, denoted by  $\{\partial_1(G),\partial_2(G),\dots,\partial_n(G)\}$, consists of the eigenvalues of $\mathcal{D}(G)$ arranged in non-increasing order, that is, $\partial_1 (G) \geq \partial_2 (G) \geq \dots \geq \partial_n (G)$.\\
An $n \times n$ matrix $A$ is said to be \textit{cogredient} to a matrix $E$ if there exists a permutation matrix $P$ such that $PAP^T=E$, where $P^T$ denotes the transpose of the matrix $P$. A matrix $A$ is called \textit{reducible} if it is cogredient to a block matrix of the form:
$$
E = \begin{bmatrix}
B & 0 \\
C & F
\end{bmatrix},
$$ where $B$ and $F$ are square matrices. Further, if $n=1$ and $A=0$, then $A$ is also considered reducible. Otherwise, $A$ is called \textit{irreducible}.
The well-known Perron-Frobenius theorem states that if $A$ is non-negative and irreducible, then the spectral radius $\rho(A)$ is a simple eigenvalue of $A$   and any eigenvalue of $A$ that has the same modulus is also simple. Moreover, $A$ admits a positive eigenvector $x$ corresponding to $\rho(A)$, and every non-negative eigenvector of $A$ is a scalar multiple of $x$. For more details about the Perron-Frobenius theory, we refer to \cite{Non-negative-matrices}.  It is easy to see that the distance matrix $\mathcal{D}(G)$ of a connected graph $G$ is irreducible. Hence, by the Perron-Frobenius theorem, we have $\partial_1(G)$ is simple and there is an eigenvector corresponding to $\partial_1(G)$ with strictly positive coordinates. 

The \textit{transmission} of a vertex $u \in V(G)$, denoted by  $\tr(u)$, is the sum of distances from $u$ to all other vertices in $G$ \cite{Distance-spectra}. That is, 
$$
\tr(u)=\sum_{ v\in V(G)}d(u,v).
$$
A graph $G$ is said to be \textit{transmission-regular} if $\tr(v)=k$  for some $k\in \mathbb{N}$ and for every vertex $v$ in $G$.

\section{A combinatorial formula for $\gamma (G)$}\label{main-results}
The following notation will be used to derive the results presented in this section. Let $$\mathcal{F}=\biggl\{x \in \mathbb{R}^n:\sum_{v \in V(G)} x_v =0 ~\mbox{and} ~\Vert x \Vert_{\infty}=1\biggl\}.$$
For a given $x \in \mathcal{F}$, define $ \gamma_x(G)=\max\limits_{uv \in E(G)}\vert x_u - x_v \vert $. In this notation, we have $\gamma(G)=\min\limits_{x \in \mathcal{F}}\gamma_x(G) $. A vector  $x \in \mathcal{F}$ is called an $\ell_\infty$-\textit{Fiedler vector} of $G$ if $\gamma(G)=\gamma_x(G)$.
\begin{theorem}\label{gamma-connected}
Let $G$ be a graph. Then $G$ is disconnected if and only if $\gamma(G)=0$.
\end{theorem}
\begin{proof}
    Let $G$ be a disconnected graph. Let $C_1$ and $C_2$ be two components of $G$. Without loss of generality,  let $|C_1|\leq |C_2|$. Consider the vector $x$

    $$   x_v =
\left\{
	\begin{array}{ll}
		1,  & \mbox{if } v \in V(C_1), \\
		-\frac{|C_1|}{|C_2|}, & \mbox{if } v \in V(C_2),\\
        0, & \mbox{if } v\notin V(C_1) \cup V(C_2).
	\end{array}
\right.
$$

Then, it is easy to see that  $x \in \mathcal{F}$ and $\gamma_x(G) =0$. So $\gamma(G)=0$.

Conversely, let $G$ be a graph with $\gamma(G)=0$. Let $x$ be an $\ell_{\infty}$-\textit{Fiedler vector}. Then $\gamma_x(G) =0$, which implies that $x_u=x_v$ whenever $uv \in E(G)$. 
If $G$ is connected, then the vector $x$ should be a multiple of the all-ones vector.  However, this contradicts the assumption that $x\in \mathcal{F} $. Therefore, $G$ is disconnected.
\end{proof}
 Let $\mathcal{D}_M(G)$ denote the maximum row sum of $\mathcal{D}(G)$. Next, we prove the main result of this section.
\begin{theorem}\label{main theorem}
    Let $G$ be a connected graph on $n$ vertices. Then
    $$
\gamma(G)=\frac{n}{\max\limits_{u \in V(G)}  \tr(u)}=\frac{n}{\mathcal{D}_M(G)}.
    $$
\end{theorem}
\begin{proof}
 Let $x \in \mathcal{F}$. Since replacing $x$ by $-x$ preserves feasibility, we may assume $x_u=1$ for some $u \in V(G)$. Let $w \in V(G)$ with $d(u,w)=s$, and let $u=z_0\sim z_1\sim z_2 \sim \dots \sim z_{s-1} \sim z_s=w$ be a shortest path. Then
 $$
\vert x_w-1\vert=\vert x_w-x_u\vert\leq \sum\limits_{j=0}^{s-1}\vert x_{z_j}-x_{z_{j+1}}\vert\leq s\gamma_x(G).
 $$
 Thus $x_w \geq 1-s\gamma_x(G)$. Consequently,
 $$
0=\sum\limits_{v \in V(G)} x_v\geq \sum\limits_{v \in V(G)} \Bigl(1- \gamma_x(G) d(u,v)\Bigr)=n- \gamma_x(G) \tr(u).
 $$
 Therefore
 $
 \gamma_x(G) \geq \frac{n}{\tr(u)}\geq \frac{n}{\mathcal{D}_M(G)}
 $.
 
 Next, let us prove the other side of the inequality. Let $u \in V(G)$ with $\tr(u)=\mathcal{D}_M(G)$. Define
 $
y_v=1 - \frac{n}{\mathcal{D}_M(G)}d(u,v)
 $.
 Then
 $$
\sum\limits_{v\in V(G)}y_v=\sum\limits_{v\in V(G)}\left(1 - \frac{n}{\mathcal{D}_M(G)}d(u,v)\right)=n-\frac{n}{\mathcal{D}_M(G)}\tr(u)=0.
 $$
Let $d=\diam(G)$. Let $a \in V(G)$ and $b \in V(G)$ with $d(a,b)=d$. For every $v$,
\begin{align*}
    d(a,v)+d(b,v) &\geq d\\
    \sum\limits_{v \in V(G)}\left(d(a,v)+d(b,v)\right)&\geq nd\\
    \tr(a)+\tr(b) &\geq nd
\end{align*}
Also $\mathcal{D}_M(G) \geq \tr(a)$ and $\mathcal{D}_M(G) \geq \tr(b)$. Thus
$
2\mathcal{D}_M(G) \geq nd
$,
and therefore
$
    \mathcal{D}_M(G) \geq \frac{nd}{2}
$.
Since $d(u,v) \leq d$ for any $v\in V(G)$, we have,
$$
y_v=1 - \frac{n}{\mathcal{D}_M(G)}d(u,v)\geq 1-\frac{nd}{\mathcal{D}_M(G)}\geq 1-2=-1.
$$
Thus $y \in \mathcal{F}$.\\
Finally, for an edge $vw$,
$
\vert d(u,v) -d(u,w) \vert \leq 1
$. Thus
\begin{align*}
    \vert y_v -y_w\vert 
    &= \frac{n}{\mathcal{D}_M(G)}\vert d(u,w) -d(u,v) \vert \\
    & \leq\frac{n}{\mathcal{D}_M(G)}.
\end{align*}
Therefore $ \gamma_y(G) \leq \frac{n}{\mathcal{D}_M(G)}$. Thus, $\gamma(G)=\frac{n}{\max\limits_{u \in V(G)}  \tr(u)}=\frac{n}{\mathcal{D}_M(G)}$.
\end{proof}
\begin{remark}
    Let $G$ be connected and $m=\vert E(G)\vert$. For each vertex $u$, a breadth-first search computes all distances from $u$, and hence $\tr(u)$, in $O(n+m)$ time. Repeating this for all vertices determines $\mathcal{D}_M(G)$ and
$
\gamma(G)=\frac{n}{\mathcal{D}_M(G)}
$
in $O(n(n+m))$ time. 
\end{remark}
Define
$$
\mathcal{M}(G)=\biggl\{u \in V(G): \tr(u)=\mathcal{D}_M(G)\biggl\}.
$$
In the following theorem, we fully characterize all $\ell_\infty$-Fiedler vectors associated with the graph $G$.
\begin{theorem}\label{set of vectors}
    For every $u \in \mathcal{M}(G)$, define
    $
x^{(u)}_v=1-\gamma(G)d(u,v)
    $ for every $v \in V(G)$.
    Then the set of $\ell_\infty$-Fiedler vectors is given by
    $$
\biggl\{\pm x^{(u)} : u \in \mathcal{M}(G) \biggr\}.
    $$
    \end{theorem}
    \begin{proof}
For $u \in \mathcal{M}(G)$, let $x^{(u)}$ be the vector defined as above. Then, by the proof of Theorem \ref{main theorem}, it is easy to see that both the vectors $x^{(u)}$ and $-x^{(u)}$ are $\ell_\infty$-Fiedler vectors. 
    
        Let $x$ be an $\ell_\infty$-Fiedler vector. Without loss of generality, let $x_u=1$ for some $u \in V(G)$. Then
        $
        x_v \geq 1- \gamma(G) d(u,v)
        $.
    Using Theorem \ref{main theorem} and summing over all vertices, we obtain
    $$
0\geq n-\gamma(G) \tr(u)=n(1-\frac{\tr(u)}{\mathcal{D}_M(G)})\geq 0.
    $$
    Therefore, equality holds everywhere. Hence 
    $
\tr(u)=\mathcal{D}_M(G)
    $
and 
$
x_v = 1- \gamma(G) d(u,v)
$
for every vertex $v \in V(G)$.
    \end{proof}

Next, we provide an equivalent condition characterizing when an $\ell_\infty$-Fiedler vector has a zero coordinate.
\begin{theorem}\label{zero-coordinate}
Let $G$ be a connected graph on $n$ vertices. Then the following statements are equivalent.
\begin{enumerate}
    \item[$(i)$] $n | D_{M}(G)$.
    \item[$(ii)$] Every $\ell_{\infty}$-Fiedler vector has a zero coordinate.
    \item[$(iii)$] There exists an $\ell_{\infty}$-Fiedler vector with a zero coordinate.
\end{enumerate}
\end{theorem}
\begin{proof}
$(i)\implies (ii)$
Let $\mathcal{D}_M(G)=nr$ for some $r \in \mathbb{N}$. Thus, by Theorem \ref{main theorem}, $\gamma(G)=\frac{1}{r}$. Let 
$$
x^{(z)}_v=1- \gamma(G)d(z,v), \qquad v\in V(G),
$$
be an $\ell_{\infty}$-Fiedler vector for some $z \in \mathcal{M}(G)$. Then
    $$
\frac{1}{n}\sum\limits_{v \in V(G)}d(z,v)=\frac{\mathcal{D}_M(G)}{n}=\frac{1}{\gamma(G)}=r.
    $$
    Note that, 
    $$
r=\frac{1}{n}\sum\limits_{v \in V(G)}d(z,v)\leq \frac{(n-1)ecc(z)}{n}<ecc(z).
    $$
    Thus $1 \leq r \leq ecc(z)-1$. Take a shortest path from $z$ to an eccentric vertex. This path contains a vertex at every distance $1,2, \dots,ecc(z)$ from $z$ and a particular vertex $y$ satisfies $d(z,y)=r$. Then
    $$
x^{(z)}_y=1-\gamma(G)d(z,y)=1-\frac{1}{r}r=0.
    $$
    Thus $x^{(z})$ has a zero coordinate. Since the choice of $x^{(z)}$ is arbitrary, every $\ell_{\infty}$-Fiedler vector has a zero coordinate.\\
    $(ii)\implies (iii)$ Trivial.\\
   $(iii) \implies (i) $ Without loss of generality, let $u \in \mathcal{M}(G)$ and the corresponding $\ell_\infty$-Fiedler vector 
    $$
x^{(u)}_v=1-\gamma(G)d(u,v), \qquad v\in V(G),
    $$
    has a zero coordinate. Then for some $w \in V(G)$, $1-\gamma(G)d(u,w)=0$. Thus $\gamma(G)=\frac{1}{d(u,w)}$. By Theorem \ref{main theorem}, we have $\mathcal{D}_M(G)=nd(u,w)$. Therefore $n | \mathcal{D}_M(G)$.\\
    
\end{proof}
We next apply the formula given in Theorem \ref{main theorem} to determine $\gamma(G)$ for some well-known graph families.
\begin{cor}\label{basic examples}
\begin{enumerate}
    \item [(a)] Let $K_n$ denote the complete graph on $n$ vertices. Then, $\gamma(K_n)=\frac{n}{n-1}$.
    \item [(b)] Let $C_n$ denote the cycle graph on $n$ vertices. Then,
    $\gamma(C_n)=\frac{n}{\left\lfloor \frac{n^2}{4}\right\rfloor}$.
    \item[(c)] Let $S_n$ denote the star graph on $n$ vertices. Then $\gamma(S_n)=\frac{n}{2n-3}.$\label{star graph}

  \item[(d)]   Let $K_{m,n}$ denote the complete bipartite graph where the partite sets have sizes $m$ and $n$ with $m \geq n$. Then $\gamma(K_{m,n})=\frac{m+n}{2m+n-2}$.

    \begin{proof}
       \textbf{$(a)$} Since every vertex in $K_n$ is adjacent to every other vertex, $\tr(u)=n-1$ for every $u \in V(K_n)$ and hence $\gamma (K_n) = \frac{n}{n-1}$.
       
       \textbf{$(b)$} Let $n=2l$ for some $l \in \mathbb{N}$ and $l \geq 2$. Let $u \in V(C_{2l})$ and $\vert S_r(u) \vert=a_r$. Then,
    $\diam(C_{2l})=l$ and $a_1=a_2=\dots=a_{l-1}=2$, $a_l=1$. Thus,
     $$
 \tr(u)=2\sum^{l-1}_{r=1}r+l=l^2.
     $$
     Hence $\gamma(C_{2l})=\frac{2l}{l^2}=\frac{2}{l}$.\\
     
     Let $n=2l+1$ for some $l \in \mathbb{N}$ and $l \geq 1$. Let $u \in  V(C_{2l+1})$. Similarly we have
      $$
 \tr(u)=2\sum^{l}_{r=1}r=l(l+1),
     $$
     for every $u\in V(C_{2l+1})$. Hence $\gamma(C_{2l+1})=\frac{2l+1}{l(l+1)}$.
     
     Therefore, combining the results for even and odd values of $n$, we obtain $\gamma(C_n)=\frac{n}{\left\lfloor \frac{n^2}{4}\right\rfloor}$.\\

\textbf{$(c)$} Let $u$ be the center vertex of $S_n$. Then $\tr(u)=n-1$. If $v$ is a pendant vertex, then $\tr(v)=2n-3$. So, $\mathcal{D}_M(S_n)=2n-3$. Thus $\gamma(S_n)=\frac{n}{2n-3}.$

\textbf{$(d)$} Let $(V_1,V_2)$ be the bipartition of $K_{m,n}$ with $\vert V_1\vert=m$ and $\vert V_2\vert=n$. Then, for any vertex $u\in V_1$, $\tr(u)=n+2m-2$. Similarly, for any vertex $v \in V_2$, $\tr(v)=m+2n-2$. Since $m\geq n$, $\mathcal{D}_M(K_{m,n})=2m+n-2$. Hence the result follows.
\end{proof}

\end{enumerate}
\end{cor}
For trees, Zelinka~\cite{Zelinka1968} proved that the maximum transmission occurs only at the pendant vertices.  Therefore, to compute $\gamma(T)$ for a tree $T$, it suffices to consider only the transmissions at the pendant vertices. Applying this observation and using Theorem \ref{main theorem} to the path $P_n$, one can directly obtain the value of $\gamma(P_n)$, which was detailed in \cite{enide-geir-2024}.
   
\begin{cor}
    Let $P_n$ denote the path graph on $n$ vertices. Then $\gamma(P_n)=\frac{2}{n-1}$.
\end{cor}
\begin{proof}
    Every pendant vertex of $P_n$ has transmission $\frac{n(n-1)}{2}$. Thus $\gamma(P_n)=\frac{2}{n-1}$. 
\end{proof}

\section{Cartesian product}\label{sec-product}

In this section, we derive a formula for $\gamma(G)$ when $G$ is the Cartesian product of finitely many graphs, expressing it in terms of the corresponding values for their constituent graphs. As a by-product of this formula, we obtain explicit values of $\gamma(G)$ for several well-known families of graphs.

The \textit{Cartesian product} of two graphs $G$ and $H$, denoted by $G\times H$, is the graph with the vertex set $V(G) \times V(H)$, and edges defined as follows: $(u_1,v_1) \sim (u_2,v_2)$ if either $u_1 \sim u_2$ in $G$ and $v_1=v_2$, or $u_1=u_2$ and $v_1 \sim v_2$ in $H$ \cite{hammack-product-graph}.

\begin{theorem}{\label{graph products}}
 
 Let $G$ and $H$ be two connected graphs. Then 
    $$
\gamma(G \times H)=\frac{1}{\frac{1}{\gamma(G)}+\frac{1}{\gamma(H)}}.
$$
\end{theorem}
\begin{proof}
Let $G$ and $H$ be graphs of order $n$ and $m$,  respectively. Let $\mathcal{D}_M(G)=t_1$, $\mathcal{D}_M(H)=t_2$. Then $\gamma(G)=n/t_1$ and $\gamma(H)=m/t_2$. It is known that \cite{hammack-product-graph}, for any two vertices $(u_1,v_1)$ and $(u_2,v_2)$ in $G\times H$, 

    $$
d_{G\times H}\{(u_1,v_1),(u_2,v_2)\}=d_G(u_1,u_2)+d_H(v_1,v_2).
    $$
    As a consequence, it is easy to see that the transmission of a vertex $(u,v)$ in $G\times H$ is
    \begin{align*}
        \tr_{G\times H}(u,v)&=m\tr_G(u)+n\tr_H(v).
            \end{align*}

To maximize this over all $(u,v)\in V(G\times H)$, we pick $u \in V(G)$ with transmission $t_1$ and $v \in V(H)$ with transmission $t_2$. Hence,  $\mathcal{D}_M(G \times H)=mt_1+nt_2$. Therefore, by Theorem \ref{main theorem},
$$
\gamma(G\times H)=\frac{mn}{mt_1+nt_2}=\frac{1}{\frac{t_1}{n}+\frac{t_2}{m}}=\frac{1}{\frac{1}{\gamma(G)}+\frac{1}{\gamma(H)}}.
$$
\end{proof}

 Similarly, the Cartesian product of the graphs $G_1,G_2,\dots,G_k$, denoted by $G_1 \times G_2 \times \dots \times G_k$, is the graph with the vertex set $V(G_1) \times V(G_2) \times \dots \times V(G_k)$, and the edge set defined as follows: $(u_1,u_2,\dots,u_k) \sim (v_1,v_2,\dots,v_k)$ if there exists exactly one index $i \in \{1,2,\dots,k\}$ such that $u_i \sim v_i$ in $G_i$, and $u_j=v_j$ for all $j \neq i$.

Building on the proof of the above theorem, we can extend the argument to establish the following result for a finite Cartesian product of graphs.

\begin{theorem}
        Let $G_1, G_2,\dots, G_k$ be a collection of connected graphs. Then
    $$
\gamma(G_1 \times G_2 \times \dots \times G_k)=\frac{1}{\frac{1}{\gamma(G_1)}+\frac{1}{\gamma(G_2)}+\dots +\frac{1}{\gamma(G_k)}}.
    $$
\end{theorem}

Using the above theorem, we can compute the explicit value of $\gamma (G)$ for several well-known families of graphs. We present a few of them here.  

\begin{enumerate}
    \item [(a)] The $t$-dimensional hypercube $Q_t$ is the Cartesian product of $t$-copies of the complete graph $K_2$. Since $\gamma(K_2)=2$, we have $\gamma(Q_t)=\frac{2}{t}$. 
    \item [(b)] The Hamming graph $H(t,s)$ is the Cartesian product of $t$-copies of the complete graph $K_s$. Hence $\gamma(H(t,s))=\frac{s}{t(s-1)}.$

  \item[(c)]  The torus grid graph $T_{m,n}$ is the Cartesian product of two cycle graphs $C_m$ and $C_n$. By Corollary \ref{basic examples}, we have   $\gamma(C_n)=\frac{n}{\lfloor \frac{n^2}{4}\rfloor}$. Thus
    $$
    \gamma(T_{m,n})=\frac{mn}{m\left\lfloor \frac{n^2}{4} \right\rfloor +n\left\lfloor \frac{m^2}{4} \right\rfloor}.
     $$
\end{enumerate}

\section{Comparison with other parameters}\label{compari-other}
For a connected graph $G$, let $\partial_1 (G)$ denote the spectral radius of the associated distance matrix $\mathcal{D}(G)$ of $G$.
 \begin{theorem}\label{max eigenvalue}Let $G$ be a connected graph on $n$ vertices. Then
    $$
    \gamma(G) \leq \frac{n}{\partial_1 (G)},
    $$
    where equality holds if and only if  $G$ is a transmission-regular graph.
\end{theorem}
\begin{proof}
 Since the largest eigenvalue of a nonnegative matrix is bounded above by the maximum row sum \cite{matrix-analysis}, the inequality follows from Theorem \ref{main theorem}. 
 
Suppose $G$ is a transmission-regular graph. Then for every vertex $u$ in $G$, we have $\tr(u)=k$ for some $k\in \mathbb{N}$. This implies that all row sums of the distance matrix $\mathcal{D}(G)$ are equal to $k$, so the all-ones vector is an eigenvector of $\mathcal{D}(G)$ corresponding to the eigenvalue $k$. Since the largest eigenvalue is bounded above by $k$, $k$ should be the largest eigenvalue. Thus $k=\mathcal{D}_M(G)= \partial_1 (G)$.

Conversely, assume that $\mathcal{D}_M(G)= \partial_1 (G)$. By the Perron-Frobenius theorem, let $z = (z_i)$ be a positive eigenvector corresponding to the eigenvalue $\partial_1 (G)$. Let $z_k=\max\limits_{i \in V(G)} z_i$. Then for the $k$-th row in $\mathcal{D}(G)$, we have
$$
\partial_1(G) z_k=\sum\limits_{j=1}^n D_{kj}z_j \leq z_k(\sum\limits_{j=1}^n D_{kj}).
$$
Hence $\partial_1(G) \leq \sum\limits_{j=1}^n D_{kj} $. Since $\mathcal{D}_M(G)= \partial_1 (G)$, we  have $\partial_1 (G) = \sum\limits_{j=1}^n D_{kj}$. Thus $z_j=z_k$ for all $j \in \{1,2, \dots,n\}$. Hence, the all-ones vector is an eigenvector corresponding to the eigenvalue $\partial_1$, that is, all the row sums of $\mathcal{D}(G)$ are equal. So, $G$ is a transmission-regular graph.
\end{proof}
\begin{remark}
A number of lower bounds for $\partial_1(G)$ have been extensively studied in the literature \cite{Distance-spectra}.  These findings can be used to derive further upper bounds for $\gamma(G)$ across various classes of graphs.
\end{remark}
For $u \in \mathcal{M}(G)$, define
$$
\Lambda(u)= \sum\limits_{v \in V(G)} \vert 1 - \gamma(G) d(u,v) \vert,
$$
and 
$$
\kappa(u)= \vert \biggl\{vw \in E(G) : d(u,v) \neq d(u,w) \biggr\}\vert.
$$
Thus  $\kappa(u)$ is the number of edges whose endpoints have different distances from $u$. Equivalently, it counts the edges joining two consecutive distance layers with respect to $u$.
\begin{theorem}
    For every connected graph $G$,
    \begin{align}\label{b(G)-inequality}
b(G) \leq \gamma(G) \min\limits_{u \in \mathcal{M}(G)} \frac{\kappa(u)}{\Lambda(u)}.
    \end{align}
    Moreover, equality holds if and only if there exists $u \in \mathcal{M}(G)$ such that 
$$
\frac{1}{\Lambda(u)}\Bigl(1-\gamma(G)d(u,v)\Bigr)_{v \in V(G)},
$$
    is an $\ell_1$-Fiedler vector of $G$. In particular, equality holds for the complete graph $K_n$.
\end{theorem}
\begin{proof}
    Fix $u \in \mathcal{M}(G)$ and consider the corresponding $\ell_\infty$-Fiedler vector 
    $$
x^{(u)}_v=1-\gamma(G)d(u,v), \qquad v \in V(G).
    $$
    Set 
    $
    y^{(u)}=\frac{x^{(u)}}{\Vert x^{(u)} \Vert_1}=\frac{x^{(u)}}{\Lambda(u)}
    $.
    Then $\Vert y^{(u)} \Vert_1=1$ and
    $$
    \sum\limits_{v \in V(G)}y^{(u)}_v=\frac{1}{\Lambda(u)}\sum\limits_{v \in V(G)}x^{(u)}_v=0.
    $$
    Hence $y^{(u)}$ is in the feasible set of the following: 
    $$
b(G)=\min\biggl\{\sum_{vw \in E(G)}\vert z_v - z_w\vert: \sum_{v \in V(G)}z_v=0~\mbox{and}~ \Vert z \Vert_{1}=1\biggl\}.
$$
Thus,
$$
b(G) \leq \sum_{vw \in E(G)} \vert y^{(u)}_v - y^{(u)}_w\vert=\frac{1}{\Lambda(u)}\sum\limits_{vw \in E(G)}\vert x^{(u)}_v - x^{(u)}_w \vert.
$$
Now, for every edge $vw$,
$
\vert d(u,v) -d(u,w)\vert \leq 1
$.
Consequently,
\begin{align*}
\vert x^{(u)}_v - x^{(u)}_w \vert&=\gamma(G)\vert d(u,v) -d(u,w)\vert\\
&=\begin{cases}
\gamma(G), & \text{if } d(u,v)\neq d(u,w),\\[4pt]
0, & \text{if } d(u,v)=d(u,w).
\end{cases}
\end{align*}
Therefore
$
\sum\limits_{vw \in E(G)}\vert x^{(u)}_v - x^{(u)}_w \vert=\kappa(u)\gamma(G)
$,
and hence 
$
b(G) \leq  \frac{\kappa(u)\gamma(G)}{\Lambda(u)}
$.
Since this holds for every vertex $u \in \mathcal{M}(G)$, minimizing over all vertices in $\mathcal{M}(G)$ proves inequality \eqref{b(G)-inequality}.\\
Equality holds precisely when equality holds in the above estimate for a minimizing vertex $u\in \mathcal{M}(G)$, that is, when $y^{(u)}$ itself attains the minimum in the definition of $b(G)$. Equivalently,
$$
y^{(u)}=\frac{1}{\Lambda(u)}\Bigl(1-\gamma(G)d(u,v)\Bigr)_{v \in V(G)},
$$
is an $\ell_1$-Fiedler vector of $G$.\\
Let $G \cong K_n$. It was shown in \cite{enide-geir-2024} that
$
b(K_n)=\frac{n}{2}
$.
Also, by Corollary \ref{basic examples}, 
$
\gamma(K_n)=\frac{n}{n-1}
$.
Now,
\begin{align*}
\Lambda(u)&= \sum\limits_{v \in V(G)} \vert 1 - \gamma(G) d(u,v) \vert\\
&=1+\sum\limits_{v \neq u}\vert 1-\frac{n}{n-1}\vert\\
&=2.
\end{align*}
Every edge incident on $u$ joins vertices at different distances from $u$, while every edge between two vertices different from $u$ joins vertices at the same distance. Therefore
$
\kappa(u)=n-1
$.
Hence 
$$
\frac{\kappa(u)\gamma(K_n)}{\Lambda(u)}=\frac{n-1}{2}\frac{n}{n-1}=\frac{n}{2}=b(K_n).
$$

\end{proof}
Let $u,v \in V(G)$. Define
$$
c_u^-(v)= \vert \biggl\{w \sim v : d(u,w)=d(u,v)-1\biggr\}\vert,
$$
and 
$$
c_u^+(v)= \vert \biggl\{w \sim v : d(u,w)=d(u,v)+1\biggr\}\vert.
$$
In other words, $c_u^-(v)$ and $c_u^+(v)$ count, respectively, the neighbors of $v$ lying in the preceding and succeeding distance layer with respect to $u$. Next, we derive an upper bound for the algebraic connectivity in terms of $\gamma(G)$.
\begin{theorem}\label{alg-theorem}
    Let $G$ be a connected graph. Then
    \begin{align}\label{alg-ineqaulity}
a(G) \leq \gamma(G)^2 \min\limits_{u \in \mathcal{M}(G)}\frac{\kappa(u)}{\sum\limits_{v \in V(G)} \Bigl(1-\gamma(G)d(u,v)\Bigr)^2}.
    \end{align}
    Equality holds if and only if there exists $u \in \mathcal{M}(G)$ such that 
 \begin{align}\label{E2}
       c_u^+(v)- c_u^-(v)=\deg(u)-a(G)d(u,v).
    \end{align}
    for every $v \in V(G)$. In particular, this bound is sharp for complete graphs.
    
\end{theorem}
\begin{proof}
     Fix $u \in \mathcal{M}(G)$ and consider the corresponding $\ell_\infty$-Fiedler vector 
    $$
x^{(u)}_v=1-\gamma(G)d(u,v), \qquad v \in V(G).
    $$
    For every edge $vw$,
$
\vert d(u,v) -d(u,w)\vert \leq 1
$.
Thus
\begin{align*}
\vert x^{(u)}_v - x^{(u)}_w \vert^2
&=\begin{cases}
\gamma(G)^2, & \text{if } d(u,v)\neq d(u,w),\\[4pt]
0, & \text{if } d(u,v)=d(u,w).
\end{cases}
\end{align*}
Now
$$
\Bigl(x^{(u)}\Bigr)^TL(G)x^{(u)}=\sum\limits_{vw\in E(G)}(x^{(u)}_v-x^{(u)}_w)^2=\kappa(u)\gamma(G)^2.
$$
Also 
$$
\Bigl(x^{(u)}\Bigr)^Tx^{(u)}=\sum\limits_{v \in V(G)} \Bigl(1-\gamma(G)d(u,v)\Bigr)^2.
$$
Since $x^{(u)} \perp \mathbf{1}$, the Rayleigh quotient formula for the algebraic connectivity gives 
$$
a(G) \leq \gamma(G)^2 \frac{\kappa(u)}{\sum\limits_{v \in V(G)} \Bigl(1-\gamma(G)d(u,v)\Bigr)^2}.
$$
Taking the minimum over $u \in \mathcal{M}(G)$ proves \eqref{alg-ineqaulity}.

For $u \in \mathcal{M}(G) $, equality in the Rayleigh quotient occurs if and only if $x^{(u)}$ is an eigenvector of $L(G)$ corresponding to eigenvalue $a(G)$. Thus equality in the stated bound holds if and only if there exists $u \in \mathcal{M}(G) $ such that
\begin{align}\label{eig-vec-equation}
    L(G)x^{(u)}=a(G)x^{(u)}.
\end{align}
Also
\begin{align}
   \Bigl( L(G)x^{(u)} \Bigr)_v
   &=\deg(v)x^{(u)}_v-\sum\limits_{w \sim v}x^{(u)}_w \notag \\
   &=\sum\limits_{w \sim v} \Bigl(x^{(u)}_v-x^{(u)}_w \Bigr).\notag
\end{align}
A neighbor of $v$ in the preceding layer contributes $-\gamma(G)$, a neighbor in the succeeding layer contributes $\gamma(G)$, and a neighbor in the same layer contributes $0$.
Thus
$$
\Bigl( L(G)x^{(u)} \Bigr)_v=\gamma(G)\Bigl(c_u^+(v)-c_u^-(v)\Bigr).
$$
Therefore the eigenvector equation \eqref{eig-vec-equation} is equivalent, at every vertex $v$, to 
\begin{align} \label{E3}
    \gamma(G)\Bigl(c_u^+(v)-c_u^-(v)\Bigr)=a(G)\Bigl(1-\gamma(G)d(u,v) \Bigr).
\end{align}
Now, we evaluate at $v=u$. Since
$
c_u^-(u)=0$ and $c_u^+(u)=\deg(u)
$,
we obtain
$
a(G)=\gamma(G)\deg(u)
$.
Substituting in \eqref{E3} gives
$$
c_u^+(v)-c_u^-(v)=\deg(u)-a(G)d(u,v),
$$
for every $v \in V(G)$.
Thus, at equality, the difference between the numbers of neighbors going away and toward a maximum-transmission vertex remains constant within each distance layer and varies linearly with distance.\\
Conversely, suppose that there exists $u \in \mathcal{M}(G)$ satisfying \eqref{E2} for every $v \in V(G)$. Every edge joining consecutive distance layers is counted once in $\sum\limits_{v \in V(G)}c_u^+(v)$ and once in $\sum\limits_{v \in V(G)}c_u^-(v)$, whereas edges within a layer contribute to neither sum. Thus summing \eqref{E2} over all vertices gives
$$
0=n \deg(u)-a(G)\tr(u).
$$
Since $u \in \mathcal{M}(G)$, Theorem \ref{main theorem} yields
\begin{align}\label{E1}
a(G)=\frac{n\deg(u)}{\tr(u)}=\gamma(G)\deg(u).
\end{align}
Multiplying \eqref{E2} by $\gamma(G)$ and using \eqref{E1},
$$
\gamma(G)\Bigl(c_u^+(v)-c_u^-(v)\Bigr)=a(G)\Bigl(1-\gamma(G)d(u,v) \Bigr),
$$
for every $v$, which is exactly 
$$
 L(G)x^{(u)}=a(G)x^{(u)}.
$$
Thus $x^{(u)}$ is a Fiedler vector, and equality holds in the Rayleigh quotient and hence the stated bound \eqref{alg-ineqaulity}.\\
Let $G \cong K_n$. It was shown in \cite{alg-con} that
$
a(K_n)=n
$.
Let $u \in V(K_n)=\mathcal{M}(K_n)$. Then $\kappa(u)=n-1$. By using Corollary \ref{basic examples}
\begin{align*}
   \sum\limits_{v \in V(G)} \Bigl(1-\gamma(G)d(u,v)\Bigr)^2&=\sum\limits_{v \in V(G)} \Bigl(1-\frac{n}{n-1}d(u,v)\Bigr)^2 \\
   &=1^2+\sum\limits_{v \neq u}\Big(1-\frac{n}{n-1}\Bigr)^2\\
   &=1+\frac{1}{(n-1)^2}(n-1)\\
   &=\frac{n}{n-1}.
\end{align*}
Thus
$$
\gamma(G)^2 \frac{\kappa(u)}{\sum\limits_{v \in V(G)} \Bigl(1-\gamma(G)d(u,v)\Bigr)^2}=\Bigl(\frac{n}{n-1} \Bigr)^2\frac{n-1}{\frac{n}{n-1}}=n=a(K_n).
$$

\end{proof}
\begin{cor}\label{integral-a(G)}
    Let $G$ be a connected graph that attains equality in \eqref{alg-ineqaulity}. Then $a(G) \in \mathbb{N} $.
\end{cor}
\begin{proof}
    By Theorem \ref{alg-theorem}, there exists  $u \in \mathcal{M}(G)$ such that 
    \begin{align*}
        c_u^+(v)-c_u^-(v)=\deg(u)-a(G)d(u,v),
    \end{align*}
    for all $v \in V(G)$. Let $w \in S_1(u)$. Then $c_u^-(w)=1$. Therefore 
    $$
c_u^+(w)-1=\deg(u)-a(G).
    $$
    Since the left-hand side is an integer, $G$ has integral algebraic connectivity.
\end{proof}

A $k$-regular graph on $n$ vertices is called strongly regular with parameters $(n,k,\lambda,\mu)$ if there are two non-negative integers $\lambda$ and $\mu$ such that any two adjacent vertices have $\lambda$ common neighbors, and any two distinct non-adjacent vertices have $\mu$ common neighbors.  We refer to \cite{BrouwerHaemers} for more details.

Let $G$ be a connected non-complete strongly regular graph. Then $\diam(G)=2$. Let $J$ be the all-ones matrix; then
\begin{align}
    A(G)^2=(k-\mu)I+(\lambda - \mu)A(G)+\mu J.
\end{align}
$k$ is always an eigenvalue of $G$ and other two eigenvalues $r,s$ satisfy 
\begin{align}\label{srg equation}
    x^2-(\lambda-\mu)x-(k-\mu)=0.
\end{align}
Hence, the eigenvalues satisfy
\begin{align}
    r+s&=\lambda-\mu \label{addition}\\
    rs&=-(k-\mu).\label{multiplication}
\end{align}
We assume that $r \geq s$ and the eigenvalues are given explicitly by
\begin{align}
    r,s=\frac{(\lambda-\mu)\pm \sqrt{(\lambda-\mu)^2+4(k-\mu)}}{2}\label{eigenvalues of srg}.
\end{align}
\begin{theorem}\label{strongly-regular-case}
    A connected non-complete strongly regular graph $G$ attains equality in \eqref{alg-ineqaulity} if and only if the least eigenvalue of $A(G)$ is $-2$.
    \begin{proof}
        Assume that $G$ attains equality in \eqref{alg-ineqaulity}. Since $\diam(G)=2$,
        $
\tr(v)=k+2(n-k-1)=2n-k-2
        $.
        Thus every vertex attains maximum transmission, and by Theorem \ref{main theorem}, we have
        $
\gamma(G)=\frac{n}{2n-k-2}
        $.\\
        Let $u \in \mathcal{M}(G)=V(G)$ satisfies \eqref{E2} for every $v \in V(G)$.
        Let $w \in S_1(u)$. Then $c_u^-(w)=1$. Every neighbor of $w$ in $S_1(u)$ is a common neighbor of $w$ and $u$. Since any two adjacent vertices have $\lambda$ common neighbors, $w$ has  exactly $\lambda$ neighbors in $S_1(u)$. Thus
        $
c_u^+(w)=k-\lambda-1
        $.
        Therefore 
        $
c_u^+(w)-c_u^-(w)=k-\lambda-2
        $.
        By using \eqref{E2}
        \begin{align*}
            k-a(G)&=k-\lambda-2\\
            a(G)&=\lambda+2.
        \end{align*}
        Also $L(G)=kI-A(G)$. As $r \geq s$, we have $a(G)=k-r$. Thus 
        \begin{align*}
            k-r&=\lambda+2\\
            \lambda&=k-r-2.
        \end{align*}
        By using \eqref{addition} and \eqref{multiplication}, we have 
        \begin{align*}
            \mu=\lambda-r-s&=k-2r-s-2\\
            \mu&=k+rs.
        \end{align*}
        Therefore
        \begin{align}
            rs+2r+s+2&=0\notag \\
            (r+1)(s+2)&=0\notag.
        \end{align}
        Since $k-\mu \geq 0$, 
        $
\sqrt{(\lambda-\mu)^2+4(k-\mu)} \geq \vert \lambda-\mu\vert
        $.
By \eqref{eigenvalues of srg}, we have $r \geq 0$ and therefore $s+2=0$, i.e. $s=-2$.\\
        Conversely, suppose $G$ is a strongly regular graph with least eigenvalue $s=-2$.
        Since $-2$  satisfies \eqref{srg equation}, we have
        \begin{align}
            4+2(\lambda-\mu)-k+\mu&=0\notag \\
            \mu&=2\lambda-k+4\label{mu}.
        \end{align}
        By \eqref{addition},
        $r-2=k-\lambda-4$. Thus
        \begin{align}\label{a(G)}
a(G)=k-r=k-(k-2-\lambda)=2+\lambda.
        \end{align}
        Let us fix any vertex $u$. We will show $u$ satisfies \eqref{E2}. 
        Since $\diam(G)=2$, we have $V(G)=\{u\} \cup S_1(u) \cup S_2(u)$.
        Note that 
        $$
c^+_u(u)-c^-_u(u)=k-0=k-a(G)d(u,u).
        $$
        Let $w \in S_1(u)$. Then by using \eqref{a(G)}
        $$
c^+_u(w)-c^-_u(w)=k-\lambda-1-1=k-a(G)d(u,w).
        $$
        Let $z \in S_2(u)$. From \eqref{mu} and \eqref{a(G)}, we get 
        $$
c^+_u(z)-c^-_u(z)=0-\mu=k-2\lambda-4=k-a(G)d(u,z).
        $$
        Thus, $u$ satisfies \eqref{E2} for every $v\in V(G)$. Therefore, $G$ attains the bound in \eqref{alg-ineqaulity}.
    \end{proof}
\end{theorem}
\begin{remark}
    Strongly regular graphs with least eigenvalue $-2$ are well-studied in the literature  \cite{Seidel-classification}.
\end{remark}

\section{A study on weighted Graphs}\label{weighted-analogue}
Let $G=\Bigl(V(G), E(G),w \Bigr)$ be a connected undirected graph on $n \geq 2 $ vertices with a strictly positive edge weight function
$
w:E(G) \rightarrow(0,\infty)
$.
Let $w_{uv}$ denote the weight of the edge $uv$. For a path $P:v_0 \sim v_1\sim \dots \sim v_k$, define its length by
$
w(P)=\sum\limits_{i=1}^kw_{v_{i-1}v_i}
$.
The weighted distance between $u$ and $v$ is 
$
d_w(u,v)
=\min\limits_{P} w(P)
$, where the minimum is taken over all paths between the vertices $u$ and $v$.
Define the weighted transmission of $u$ by 
$
\tr_w(u)=\sum\limits_{v \in V(G)}d_w(u,v)
$.
 We define the weighted analog of $\gamma(G)$ as follows:
 $$
\gamma_w(G)=\min\biggl\{\max\limits_{uv \in E(G)}\frac{\vert x_u-x_v\vert}{w_{uv}}:\sum\limits_{v \in V(G)}x_v=0 ~\mbox{and} ~\Vert x \Vert_{\infty}=1\biggr\}. 
 $$
This quantity $\gamma_w(G)$ represents the smallest possible maximum variation of the vertex values along edges per unit edge weight, among all vertex assignments having mean $0$ and unit $\ell_{\infty}$-norm.\\
Let $\mathcal{D}_M^w(G)= \max\limits_{u\in V(G)}\tr_w(u)$ and 
$$
\mathcal{M}_w(G)=\biggl\{u \in V(G): \tr_w(u)= \mathcal{D}_M^w(G)\biggr\}.
$$
We conclude this section by presenting weighted analogues of Theorems \ref{main theorem} and \ref{set of vectors} without proof. The proof techniques are similar to that of unweighted case.
\begin{theorem}\label{weighted-main}
   Let $G=\Bigl(V(G), E(G),w \Bigr)$ be a connected undirected graph on $n \geq 2 $ vertices with positive edge weights. Then
   $$
\gamma_w(G)=\frac{n}{\max\limits_{u\in V(G)}\tr_w(u)}=\frac{n}{\mathcal{D}_M^w(G)}.
   $$
\end{theorem}
\begin{theorem}
For $u \in \mathcal{M}_w(G)$, define
$
f^{(u)}_v=1-\gamma_w(G)d_w(u,v) 
$.
    The full set of weighted $\ell_{\infty}$-Fiedler vectors is 
    $$
\biggl\{\pm f^{(u)}: u \in \mathcal{M}_w(G) \biggr\}.
    $$
   
\end{theorem}

\section{Conclusion}\label{conclusion}
In this work, we studied the parameter $\gamma(G)$ as an $\ell_\infty$-analogue of the classical algebraic connectivity and demonstrated that $\gamma(G)$ characterizes the connectedness of the underlying graph $G$. We established a combinatorial formula for $\gamma(G)$ based on the maximum row sum of the distance matrix of $G$, which enables efficient computation of $\gamma(G)$ using a breadth-first search algorithm.  It also yields a complete description of all $\ell_{\infty}$-Fiedler vectors as functions of distances from maximum-transmission vertices. We  used this description to characterize the graphs having $\ell_{\infty}$-Fiedler vectors with zero coordinates. We derived a formula for $\gamma(G)$ when $G$ is the Cartesian product of finitely many graphs, and calculated $\gamma(G)$ for some well-known graphs. Furthermore, we derived comparisons with the distance spectral radius, $b(G)$, and the algebraic connectivity, and investigated the corresponding equality cases. The transmission formula and the description of optimal vectors extend to weighted graphs.\\
The equality problem for the algebraic connectivity bound deserves further investigation. Theorem \ref{alg-theorem} gives necessary and sufficient conditions in terms of distance layers, while Corollary \ref{integral-a(G)} shows that equality forces integral algebraic connectivity. Within the class of connected non-complete strongly regular graphs, Theorem \ref{strongly-regular-case} identifies the equality graphs precisely as those with least adjacency eigenvalue $-2$. These results motivate the following problem.
 Characterize all connected graphs that attain equality in the bound of Theorem \ref{alg-theorem}.

 We conclude with the following table summarizing the graph smoothing problems with respect to the $\ell_2$-norm, $\ell_1$-norm, and $\ell_\infty$-norm. The conclusions in the last column are based on the results presented in this work.

\begin{table}[H]
\centering
\Large 
\resizebox{\textwidth}{!}{%
\begin{tabular}{|c|c|c|c|}
\hline
\textbf{Name} &  $\ell_2$-graph smoothing  & $\ell_1$-graph smoothing & $\ell_\infty$-graph smoothing\\
\hline
\textbf{} &   &  & \\
\textbf{Problem} & $\min\limits_x \sum\limits_{uv \in E(G)} (x_u - x_v)^2$ & $\min\limits_x \sum\limits_{uv \in E(G)} \lvert x_u - x_v \rvert$ & $\min\limits_x \max\limits_{uv \in E(G)} \lvert x_u - x_v \rvert$ \\
 & subject to $\| x \|_2 = 1$, $x \perp \mathbf{1}$ & subject to $\| x \|_1 = 1$, $x \perp \mathbf{1}$ & subject to $\| x \|_\infty = 1$, $x \perp \mathbf{1}$ \\
 \textbf{} &   &  & \\
 \hline
\textbf{Solution} & $a(G)$ & $b(G)$ & $\gamma (G)$ \\
\hline
\textbf{Encodes} & Yes & Yes & Yes \\
\textbf{the connectedness} &  &  &  \\
\hline
\textbf{Relevance} & partition & Sparsest cut &  max transmission \\
\hline
\textbf{Strategy} & spectral & combinatorial & combinatorial \\
 &  & (NP-hard) & (Polynomial-time) \\
\hline
\end{tabular}%
}
\caption{\text{Comparison of graph smoothing problems under different norms.}}
\label{tab:graph_smoothing}
\end{table}
\section*{Acknowledgments}
M.~Rajesh Kannan acknowledges financial support from ANRF-CRG, India (File No.\ CRG/2023/002747). Rahul Roy thanks the University Grants Commission (UGC), India, for financial support (NTA Ref. No. 231610209574).

\section*{Declaration of AI use}
The authors acknowledge the use of ChatGPT for improving the exposition of the manuscript and refining some of the proofs. All mathematical statements have been independently verified by the authors, who are solely responsible for the correctness of the results presented in this paper.

\bibliographystyle{amsplain}
	\bibliography{gamma(X)_ref}

\providecommand{\bysame}{\leavevmode\hbox to3em{\hrulefill}\thinspace}
\providecommand{\MR}{\relax\ifhmode\unskip\space\fi MR }
\providecommand{\MRhref}[2]{%
  \href{http://www.ams.org/mathscinet-getitem?mr=#1}{#2}
}
\providecommand{\href}[2]{#2}
\begin{thebibliography}{10}

\bibitem{enide-geir-2024}
Enide Andrade and Geir Dahl, \emph{Combinatorial {F}iedler theory and graph
  partition}, Linear Algebra Appl. \textbf{687} (2024), 229--251. \MR{4708740}

\bibitem{Distance-spectra}
Mustapha Aouchiche and Pierre Hansen, \emph{Distance spectra of graphs: a
  survey}, Linear Algebra Appl. \textbf{458} (2014), 301--386. \MR{3231823}

\bibitem{graphs-and-matrices}
Ravindra~B. Bapat, \emph{Graphs and matrices}, second ed., Universitext,
  Springer, London; Hindustan Book Agency, New Delhi, 2014. \MR{3289036}

\bibitem{Non-negative-matrices}
Abraham Berman and Robert~J. Plemmons, \emph{Nonnegative matrices in the
  mathematical sciences}, Classics in Applied Mathematics, vol.~9, Society for
  Industrial and Applied Mathematics (SIAM), Philadelphia, PA, 1994, Revised
  reprint of the 1979 original. \MR{1298430}

\bibitem{BrouwerHaemers}
A.~E. Brouwer and W.~H. Haemers, \emph{Spectra of graphs}, Universitext,
  Springer, New York, 2012. \MR{2882891}

\bibitem{FernandezGoycooleaEtAl2026}
Jos{\'e} Fern{\'a}ndez~Goycoolea, Andrea de~las Heras-Parrilla, Luis~H.
  Herrera, Clemens Huemer, and Carlos Seara, \emph{A geometric view of
  combinatorial {Fiedler} theory}, arXiv preprint arXiv:2607.00519, 2026.

\bibitem{alg-con}
Miroslav Fiedler, \emph{Algebraic connectivity of graphs}, Czechoslovak Math.
  J. \textbf{23(98)} (1973), 298--305. \MR{318007}

\bibitem{fiedler-1989}
\bysame, \emph{Laplacian of graphs and algebraic connectivity}, Combinatorics
  and graph theory ({W}arsaw, 1987), Banach Center Publ., vol.~25, PWN, Warsaw,
  1989, pp.~57--70. \MR{1097636}

\bibitem{ratio-bound}
Willem~H. Haemers, \emph{Hoffman's ratio bound}, Linear Algebra Appl.
  \textbf{617} (2021), 215--219. \MR{4218548}

\bibitem{hammack-product-graph}
Richard Hammack, Wilfried Imrich, and Sandi Klav\v{z}ar, \emph{Handbook of
  product graphs}, second ed., Discrete Mathematics and its Applications (Boca
  Raton), CRC Press, Boca Raton, FL, 2011, With a foreword by Peter Winkler.
  \MR{2817074}

\bibitem{matrix-analysis}
Roger~A. Horn and Charles~R. Johnson, \emph{Matrix analysis}, second ed.,
  Cambridge University Press, Cambridge, 2013. \MR{2978290}

\bibitem{mohar-lap-sur-1}
Bojan Mohar, \emph{The {L}aplacian spectrum of graphs}, Graph theory,
  combinatorics, and applications. {V}ol.\ 2 ({K}alamazoo, {MI}, 1988),
  Wiley-Intersci. Publ., Wiley, New York, 1991, pp.~871--898. \MR{1170831}

\bibitem{mohar-eigen-comb-opti}
Bojan Mohar and Svatopluk Poljak, \emph{Eigenvalues in combinatorial
  optimization}, Combinatorial and graph-theoretical problems in linear algebra
  ({M}inneapolis, {MN}, 1991), IMA Vol. Math. Appl., vol.~50, Springer, New
  York, 1993, pp.~107--151. \MR{1240959}

\bibitem{nica-book-2018}
Bogdan Nica, \emph{A brief introduction to spectral graph theory}, EMS
  Textbooks in Mathematics, European Mathematical Society (EMS), Z\"urich,
  2018. \MR{3821579}

\bibitem{Seidel-classification}
J.~J. Seidel, \emph{Strongly regular graphs with {$(-1,\,1,\,0)$} adjacency
  matrix having eigenvalue {$3$}}, Linear Algebra Appl. \textbf{1} (1968),
  281--298. \MR{234861}

\bibitem{trevisan_spec-partition}
Luca Trevisan, \emph{Lecture notes on graph partitioning expanders and spectral
  methods}, University of California Lecture Notes, Berkeley (2017).

\bibitem{ulrike-spectral}
Ulrike von Luxburg, \emph{A tutorial on spectral clustering}, Stat. Comput.
  \textbf{17} (2007), no.~4, 395--416. \MR{2409803}

\bibitem{clique-indep-bound-wilf}
Herbert~S. Wilf, \emph{Spectral bounds for the clique and independence numbers
  of graphs}, J. Combin. Theory Ser. B \textbf{40} (1986), no.~1, 113--117.
  \MR{830598}

\bibitem{Zelinka1968}
Bohdan Zelinka, \emph{Medians and peripherians of trees}, Archivum Mathematicum
  \textbf{4} (1968), no.~2, 87--95. \MR{0269541}

\end{thebibliography}

    \vspace{0.4cm}

\affl{M. Rajesh Kannan}{rajeshkannan@math.iith.ac.in, rajeshkannan1.m@gmail.com}{Department of Mathematics, Indian Institute of Technology Hyderabad, Kandi, Sangareddy 502284, India.}

\affl{Rahul Roy}{ma23resch11004@iith.ac.in, rahulroy13832@gmail.com}{ Department of Mathematics, Indian Institute of Technology Hyderabad, Kandi, Sangareddy 502284, India.}

\end{document}